\documentclass[12pt,leqno]{amsart}
\usepackage[dvips]{graphics}
\usepackage{amssymb}

\numberwithin{equation}{section}

\newtheorem{thm}{Theorem}[section]

\newtheorem{lem}[thm]{Lemma}

\newtheorem{prop}[thm]{Proposition}

 \theoremstyle{definition}

\theoremstyle{remark}

\newcommand{\tref}[1]{Theorem~\ref{#1}}
\newcommand{\cref}[1]{Corollary~\ref{#1}}

\newcommand{\R}{\mathbb{R}}

\newcommand{\red}[1]{{\color{red}}}

\begin{document}
\pagebreak


\title{Hilbert space factor of metric spaces}







\author{Thomas Foertsch, Alexander Lytchak and Elefterios Soultanis}



\begin{abstract}
We prove that any complete metric space has a unique 
decomposition as a direct product of a possibly finite or zero-dimensional Hilbert space and a space that does not split off lines.	
\end{abstract}

\maketitle
\renewcommand{\theequation}{\arabic{section}.\arabic{equation}}
\pagenumbering{arabic}



\section{Introduction}
A \emph{line} $\ell$ in metric space $X$ is a subset isometric to the real line  $\R$.
We call a line $\ell \subset X$ a \emph{splitting line}  in $X$  if there exists a metric space $Y$ and an isometry 
$\Phi: X\to \R \times Y$ onto the direct product, such that the composition of $\Phi$ and the projection $P^{Y}$ to $Y$
sends $\ell$ to a point.

In this note we verify the following result:
\begin{thm} \label{thm: main}
	Let $X$ be a complete metric space. Then there exists a unique decomposition of $X$ as a direct product $X=Y\times H$, where the metric space
	$Y$ does not contain splitting lines  and $H$ is a (possibly finite-  or zero-dimensional) Hilbert space.
	
	For any point $x\in X$, the maximal Hilbert space factor $H_x$ (with respect to inclusion) containing $x$ coincides with the union of all splitting lines $\ell \subset X$ containing $x$.
%
\end{thm}

As a consequence,  the isometry group $Iso(X)$ of $X$ is canonically isomorphic to
$Iso (Y)\times Iso(H)$, cf. \cite[Section 2.5]{FL}. 

 Theorem  \ref{thm: main} contains no assumptions on the space $X$ other than completeness. Nevertheless,  the proof  is not completely trivial. The theorem is  related to a problem
 formulated in a special form by Ulam, \cite{Ulam}, asking about the uniqueness of decompositions of a metric space as a product of some \emph{irreducible factors}.  For non-complete spaces there exist surprisingly simple counterexamples to uniqueness, \cite{four}, \cite{intui}. For  compact spaces this problem is widely open.

 Under additional geometric assumptions, Theorem \ref{thm: main} is known.
   For geodesic metric spaces satisfying a finite-dimensionality assumption, a stronger  uniqueness of  product decompositions was proved in \cite{FL}, generalizing the classical de Rham decomposition theorem in Riemannian geometry
 \cite{rham}, \cite{eh}.  
 
  If the space $X$ is a CAT(0) space, then Theorem \ref{thm: main} is contained in \cite[Theorem 6.15]{BH}. Moreover, in this case  the Hilbert space factor $H$ is identified with the group of all \emph{Clifford translations} of $X$. 
 More generally, for CAT$(\kappa)$ spaces $X$ with any 
$\kappa$, the Hilbert space factor is closely related to the space of continuous affine functions on $X$, \cite{LS}.

 If $X$ is non-negatively curved in the sense of Alexandrov, then any line in $X$ is a splitting line,
   by  Toponogov's theorem \cite{Mits}, \cite[Section 16G]{AKP}. 
  In this   case, Theorem \ref{thm: main} states that the union $H_x$  of all lines through $x\in X$ is the maximal  Hilbert space  factor of $X$. 
 
 Due to the splitting theorem proven by Gigli in \cite{Gigli, Gigli1}, the last statement applies without changes to RCD$(0, N)$ spaces, $N<\infty$.

 Theorem \ref{thm: main} was obtained many years ago during the work on \cite{FL}. We were not aware of any consequence until recently, when the result appeared to play some role in the study of universal infinitesimal Hilbertianity \cite{Gigli2}.  See also \cite[Theorem A]{Che} for another recent appearance of the Hilbert space factor.
 The current paper is published jointly with the third named author who obtained an independent proof of a  version of Theorem \ref{thm: main}.  
 
 {\bf Acknowledgements.}
We would like to thank  Nicola Gigli for helpful discussions.  The paper was finalized  during the stay of Alexander Lytchak and Elefterios Soultanis at the HIM  in Bonn as part of 
 the special trimester on Metric Analysis  funded by the  DFG under Germany's Excellence Strategy – EXC-2047/1 – 390685813. Elefterios Soultanis was supported by the Research Council of Finland grant 355122.


\section{Notation}
\label{subsec-notations}
By $d$ we will denote distances in metric spaces without reference
to the space. 
For metric spaces $Y,\bar Y$, the product $Y\times \bar Y$ will always denote the direct (Cartesian) product of the metric spaces. For a direct product $X=Y\times \bar Y$ we denote
by $P^Y:X\to Y$ and $P^{\bar Y} :X\to \bar Y$  the canonical
projections onto the factors.

  For $x\in X$ we denote by $Y_x:=(P^{\bar Y})^{-1} (P^{\bar Y} (x))$ the $Y$-fiber through $x$. The restriction
$P^Y :Y_x \to Y$ is an isometry. We will  identify $Y_x$ with
$Y$ via this isometry. Using this identification, the projection $P^Y:X\to Y$ is identified with the closest-point projection $\pi_{Y_x}$ which sends points $p\in X$ to the unique point $\pi_{Y_x}(p)\in Y_x\simeq Y$ minimizing $d(p,Y_x)$. The isometry $X\to Y\times \bar Y$ is then identified with the map 
$$(\pi_Y,\pi_{\bar Y}) :X\to Y \times \bar Y \subset X\times X\;.$$
The map is surjective and satisfies for all $x_1,x_2\in X$:
$$d^2(x_1,x_2)= d^2 (\pi_Y(x_1),\pi _Y(x_2)) +d^2 (\pi _{\bar Y}(x_1),\pi _{\bar Y}(x_2))\, .$$

\section{Main observation}
\subsection{Factor subsets}
Fix  a point $x$ in a metric space $X$.   We call a subset $A$ containing $x$ a  \emph{factor subset} if there  exists an isometry  $I:X\to Y\times \bar Y$, such that $A$ coincides with $Y_x$ under this isometry.     

Any factor subset is closed.  If $x\in A$ is a factor subset then the closest-point projection  $\pi_A:X \to A$ is well-defined, the preimage $B:=\pi_A^{-1} (x)$ is another factor subset  of $X$ and  
$$(\pi_A,\pi_B):X\to A\times B$$
is an isometry.  We call $B$ the \emph{cofactor} of $A$.  

We denote by $\mathcal F_x$ the set of all factor subsets of $X$ containing $x$.  The set $\mathcal F_x$ is partially ordered by inclusion.  The map $\perp: \mathcal F_x \to \mathcal F_x$ which sends a factor subset $A\in \mathcal F_x$ to its complementary factor subset  is an order-reversing involution of $\mathcal F_x$.

\subsection{Intersections and projections of factor subsets}

The question whether the intersection of two factor subsets is again a factor subset is closely related to the  uniqueness of product decompositions, cf.    \cite[Lemma 5.1]{FL}. The following observation provides a sufficient criterion for an affirmative answer to this question:

\begin{lem} \label{lem: factorsubset}
 Let $X$ be a metric space and $x\in X$. Let $X=Z\times \bar Z$ be a decomposition and let $Y\in \mathcal F_x$ be some factor subset of $X$.
For $p\in Y$, we set $G_p:=\pi_{Y}  (\bar Z_p)$ and $F_p:=Y\cap Z_p$.
Assume that,  for all $p,q\in Y$:
\begin{enumerate}
\item The assumption $q\in G_p$ implies $G_q=G_p$.
\item  $G_p \cap F_q$ contains   exactly one point. 
\end{enumerate}
Then $F_x$ and $G_x$ are factor subsets of $Y$  and  $Y =F_x\times G_x$. 
\end{lem}

\begin{proof}
The subsets $F_p, p\in Y$ constitute a disjoint decomposition of $Y$. Due to (1),   the subsets $G_p,p\in Y$ constitute a disjoint decomposition of $Y$ as well.
By the  second assumption, we obtain for all $p_1,p_2\in Y$   a canonical map $I_{p_1,p_2}:F_{p_1}\to F_{p_2}$, which sends $q\in F_{p_1}$ to 
$G_q\cap F_{p_2}$.

By definition, $I_{p,p}$ is the identity, for all $p\in Y_x$. By the first assumption,  $I_{p_2,p_3} \circ I_{p_1,p_2}  =I_{p_1,p_3}$, for
all $p_1,p_2,p_3 \in Y_x$. 

 Due to  \cite[Subsection 2.2]{FL}, the subsets $F_p$ are factor subsets of $Y$ if, for all $p_1,p_2\in Y$, we have
\begin{equation} \label{eq: foer}
d^2(p_1,p_2)=d^2(p_1,I_{p_1,p_2} (p_1)) +d^2(I_{p_1,p_2} (p_1),p_2)\,.
\end{equation}
If this holds true, then  $G_x$ is the cofactor subset of $F_x$ in $Y$ through the point $x$.  It remains to verify equality  \eqref{eq: foer}, essentially contained in  \cite[Lemma 2.1]{FL}; we present the computation here for convenience.

Set $q=I_{p_1,p_2} (p_1)$.    By assumption, $G_q=G_{p_1}$, hence 
we find some $\bar z \in \bar Z _q$, such that $p_1=\pi _Y (\bar z)$.  Then
$$d^2(p_2,\bar z)-d^2(q,p_2)=d^2(q,\bar z)=d^2(\bar z,p_1)+d^2(p_1,q)\;,$$
$$d^2(p_2,p_1)+d^2(p_1,\bar z)=d^2(p_2,\bar z)\;.$$
Taking the sum of this two equalities we deduce 
$$d^2(p_1,p_2)-d^2(q,p_2)=d^2(p_1,q)\,,$$
hence \eqref{eq: foer}.  This finishes the proof.
\end{proof}

\subsection{Splitting lines}
We  can verify the assumptions  of Lemma \ref{lem: factorsubset} in a special case:

\begin{prop} \label{prop: mainprop}
Let $x$ be a point in a metric space $X$.  Let $Y\in \mathcal F_x$ be a factor subset and let
$\ell \in \mathcal F_x$ be a splitting line.  Then $\pi_Y(\ell)$ is either a point or a splitting line in $Y$.
\end{prop}

\begin{proof}
Denote by $\bar Y \in \mathcal F_x$ the cofactor of $Y$. 
 For $p\in Y$, denote  by $\ell_p$ the splitting line through $p$ determined by $\ell$ and by 
$Z_p$ the cofactor subset of $\ell _p$ through $p$.   As in Lemma \ref{lem: factorsubset},
we set $G_p:= \pi _{Y} (\ell _p)$ and $F_p:=Z_p \cap Y$. For any $p$, parametrize $\ell _p$ by an isometry $\gamma _p:\mathbb R\to \ell _p$, such that $\pi_{\ell} \circ \gamma_p $ is  a translation.
Then $\pi _Y \circ \gamma _p$ and $\pi _{\bar Y} \circ \gamma _p$ are geodesics parametrized with constant velocities $\alpha=\alpha (p)$ and $ \beta=\beta (p)$ such that $\alpha ^2+\beta ^2=1$,  \cite[Section 2.4]{FL}.

 For any $p,q\in Y$, the value
$$\ d^2(\gamma _p (t), \gamma _q (t)) =d^2 (\pi _{ Y}( \gamma _p(t)), \pi_{ Y} (\gamma _q(t)))+  d^2 (\pi _{\bar Y}( \gamma _p(t)), \pi_{\bar Y} (\gamma _q(t)))$$
is independent of $t$.
Hence, the first summand    $d^2(\pi_Y\circ  \gamma _p (t), \pi_y \circ \gamma _q (t))$   is bounded.  
 Therefore, the geodesics $\pi_Y\circ \gamma _p$ and $\pi _Y \circ \gamma _q$ have the same velocity  $\alpha (p)=\alpha (q)$. Hence, $\alpha$ is independent of $p$.

The projection $G_x=\pi_Y(\ell_x)$ is a point if and only if $\alpha (x)=0$.   From now on, let $G_x$ not be a single point. Then, for all $p\in Y$,  the subset $G_p$ is a line parametrized by
$\pi_Y\circ \gamma _p$ with constant velocity $\alpha\neq 0$.

In order to prove that $G_x$ is a splitting line in $Y$, we only need to verify properties (1) and (2) from Lemma \ref{lem: factorsubset}.

Let $p,q\in Y$ be such that $q\in G_p$.  In order to verify  that $G_p$ and $G_q$ coincide, we 
 consider the image $K=\pi _{Z_p} (G_p)$.  This is a geodesic (possibly a point) containing $\pi_{Z_p} (p)$ and $\pi_{Z_p} (q)$.
Hence 
$$K\times \ell _p\subset Z_p \times \ell _p =X$$ is a Euclidean plane (possibly a line), which contains $\ell_p$ and $\ell _q$ as parallel lines.   Applying \cite[Proposition 1.4]{FL}, we deduce that 
$\pi_Y(K)$ is isometric to a linearly convex subset of a normed vector space, such that 
$\pi_Y:K\to \pi_Y(K)$ is an affine map with respect to this linear structure (in fact, one can apply the more special \cite[Proposition 1]{fs}, to see that $\pi_Y(K)$ is  a normed space).  Hence, $\ell_p$ and $\ell_q$ are mapped by $\pi_Y$ to parallel lines in the subset  $\pi_Y(K)$ of a normed vector space.  Since these images $G_p$ and $G_q$ intersect in $q$, these lines have to coincide.
This proves (1).

 The statement that $G_p$ and $F_q$ intersect in exactly one point is equivalent to the statement that $\pi_{\ell} :G_p \to \ell$ is a bijection.  Then the unique intersection point will be the only point $o\in G_p$, for which $\pi _{\ell}(o)= \pi_{\ell} (q)$.

 Applying  \cite[Section 2.4]{FL} again, we see that the restriction of the projection $\pi_{\ell}$ onto the line $G_p$ is either constant or a bijective map onto a line. Since the only line contained in $\ell$ is $\ell$ itself, we deduce that either $\pi_{\ell}:G_p\to \ell$ is a bijection or constant. However, if $\pi_{\ell} (G_p)$ is a point, then $G_p$ is contained in $Z_p$. 
But then $p$ is the closest point on $G_p$ to any point on $\ell _p$.   Since $G_p$ is defined as the  closest-point projection of $\ell_p$ onto $Y$, this would imply $G_p=\{p\}$, in contradiction to our assumption that $G_x$ and hence $G_p$ are not singletons.  This  finishes the proof of    (2) and  of the Lemma.
\end{proof}

\section{Zorn's Lemma}
Consider the subset $\mathcal F^0_x$ consisting of all factor subsets in $\mathcal F_x$ which are isometric to  Hilbert spaces.        We verify:
\begin{lem} \label{lem: zorn}
If $X$ is complete, there exists a maximal element in $\mathcal F^0_x$.
\end{lem}

\begin{proof}
By Zorn's Lemma we only  need to show that for any increasing chain $I\subset \mathcal F^0_x$, there exists an element $H\in \mathcal F^0_x$ containing all $A\in I$.  We let $H$ be the closure
$$H:=\overline {\bigcup _{A\in I} A}\;.$$
The union of a chain of Hilbert spaces is a pre-Hilbert space. Since $X$ is complete,  the closure $H$ is a Hilbert space.  Clearly, any $A\in I$ is a subset of $H$.  It remains to prove that $H$ is a factor subset of $X$.

For any $p\in X$ and any $A^{\alpha} \in I$ with cofactor $B^{\alpha}$,  we set 
$$p^{\alpha}:=\pi _{A^{\alpha}} (p) \in A^{\alpha} \;\; \text{ and} \; \; p_{\alpha}:=\pi _{B^{\alpha}} (p) \in B^{\alpha}$$

For any  $A^{\alpha} \subset A^{\beta}$ in $I$, $A^{\alpha}$ is a Hilbert
 subspace of the Hilbert space $A^{\beta}$ hence a factor subset of $A^{\beta}$. Denote by 
$A^{\alpha, \beta}$ the orthogonal complement of $A^{\alpha }$ in $A^{\beta}$. Thus, 
$A^{\beta} = A^{\alpha} \times A^{\alpha,\beta}\,.$ 
Then, for all $p\in X$,
$$\pi _{A^{\alpha}} (p^{\beta} )=p^{\alpha}\; \; \text{and} \; \;
d^2(p,p^{\beta})+ d^2(p^{\alpha}, p^{\beta})= d^2(p, p^{\alpha}) \,.$$
We deduce that the net  of real numbers  $d(p,p^{\alpha})$ is monotonically decreasing, hence convergent.
Moreover, the net of points $p^{\alpha}$ in $X$ is Cauchy, hence convergent by the completeness of $X$.  The limit point $p^{\infty}$ is  contained in $H$ and it is the closest-point to $p$  on $H$.

Denote by $ B^{\alpha}, B^{\beta} \in \mathcal F_x$ the cofactors of $A^{\alpha}$ and $A^{\beta}$, respectively.  Then we have a canonical isomorphism 
$$ B^{\alpha}=A^{\alpha,\beta}\times B^{\beta }\,.$$  
As above, for any $p\in X$ the net of numbers  $d(p,p_{\alpha})$ is monotonically increasing.
Since all these numbers are bounded by $d(p,x)$, this net of numbers is converging. Then,
by completeness of $X$, the net of points $p_{\alpha}$ is convergent to a point $p_{\infty}\in X$, contained in 
$$B^{\infty} :=\bigcap _{A^{\alpha}\in I} B^{\alpha}\;.$$

Hence, the isometric embeddings 
$$(\pi _{A^{\alpha}}, \pi _{B^{\alpha}}):X\to X\times X$$
converge pointwise to a map
$$(\pi_H, \pi_{B^{\infty}}):X\to X\times X\;.$$
By continuity, this map is again an isometric embedding. Moreover, by  continuity the image
coincides with $H\times B^{\infty} \subset X\times X$.  This proves that $H$ is a factor subset and finishes the proof. 
\end{proof}

\section{End of Proof}
Now we can easily provide

\begin{proof}[Proof of Theorem \ref{thm: main}]
Fix a point $x\in X$.  By Lemma  \ref{lem: zorn}, there exists a maximal factor subset $H_x=H\in \mathcal F^0_x$ isometric to a Hilbert space.  Let $Y\in \mathcal F_x$ denote the cofactor subset, thus $X=Y\times H$. If $Y$ contained a splitting line $\ell$ then $Y=Z\times \ell$ and $X=Z\times \ell \times H$.  Then $(\ell \times H)_x$ is a factor subset in $\mathcal F_x^0$, which contains $H$ and is not equal to $H$ in contradiction to the maximality of $H$.   Thus, $Y$ contains no splitting lines.

We claim that $H$ is the union of all splitting lines in $X$ through the point $x$. This  immediately implies the uniqueness statement.
	
Any line $\ell$ in the Hilbert space $H$ through the point $x$ is a splitting line in $H$. Since $H$ is a factor subset of $X$, 
the line $\ell$ is a splitting line in $X$. On the other hand, let  $\ell \in \mathcal F_x$ be a splitting line  in $X$ and assume 	 that  $\ell$ is not contained  in  $H$.   Then $\pi_{Y} (\ell)$ is not a point, hence, by Proposition \ref{prop: mainprop}, this projection is a splitting line in $Y$. This  contradiction  finishes the proof of the claim and of the theorem.
%
\end{proof}






\bibliographystyle{alpha}
\bibliography{Hilbert}

\end{document}